\documentclass[12pt,reqno]{amsart}

\usepackage{amssymb}

\newcommand{\del}{\delta}
\newcommand{\kap}{\kappa}
\newcommand{\lam}{\lambda}
\renewcommand{\phi}{\varphi}
\newcommand{\sig}{\sigma}

\newcommand{\Del}{\Delta}

\newcommand{\R}{{\mathbb R}}

\newcommand{\M}{{\mathsf M}}

\newcommand{\od}{{\overline d}}

\newcommand{\lfl}{\left\lfloor}
\newcommand{\rfl}{\right\rfloor}
\newcommand{\<}{\langle}
\renewcommand{\>}{\rangle}

\newcommand{\longc}{,\dotsc,}
\newcommand{\longp}{+\dotsb+}
\newcommand{\longe}{=\dotsb=}
\newcommand{\longge}{\ge\dotsb\ge}

\newcommand{\est}{\varnothing}
\newcommand{\seq}{\subseteq}
\newcommand{\stm}{\setminus}

\newcommand{\sub}[1]{_{\substack{#1}}}

\newcommand{\lmax}{\lam_{\max}}

\theoremstyle{plain}
\newtheorem{lemma}{Lemma}
\newtheorem{claim}{Claim}
\newtheorem{theorem}{Theorem}
\newtheorem{corollary}{Corollary}

\newtheorem{primetheorem}{Theorem}

\newcommand{\refe}[1]{\eqref{e:#1}}
\newcommand{\refc}[1]{\ref{c:#1}}
\newcommand{\refl}[1]{\ref{l:#1}}
\newcommand{\refp}[1]{\ref{p:#1}}
\newcommand{\refs}[1]{\ref{s:#1}}
\newcommand{\reft}[1]{\ref{t:#1}}

\author{Vsevolod F. Lev}
\address{Department of Mathematics, The University of Haifa at Oranim,
  Tivon 36006, Israel}
\email{seva@math.haifa.ac.il}
\title[Largest eigenvalue and bi-average degree]%
  {The largest eigenvalue \\ and bi-average degree of a graph}

\begin{document}
\baselineskip=16pt

\begin{abstract}
We show that for a graph $G$ with the vertex set $V$ and the largest
eigenvalue $\lmax(G)$, letting
  $$ \M(G):=\max_{\est\ne X,Y\seq V} \frac{e(X,Y)}{\sqrt{|X||Y|}} $$
(where $e(X,Y)$ denotes the number of edges between $X$ and $Y$), we have
  $$ \textstyle
         \M(G) \le \lmax(G) \le \big(\frac14\log|V|+1\big)\,\M(G). $$
Here the lower bound is attained if $G$ is regular or bi-regular, whereas the
logarith\-mic factor in the upper bound, conjecturally, can be improved
--- although we present an example showing that it cannot be replaced with a
factor growing slower than $(\log |V|/\log\log|V|)^{1/8}$.

Further refinements are established, particularly in the case where $G$ is
bipartite.
\end{abstract}

\maketitle

\section{Background and summary of results: the general case}

For a graph $G$, by $\lmax(G)$ we denote the largest eigenvalue of $G$ (often
referred to as the \emph{first} eigenvalue), by $\od(G)$ the average degree
of $G$, and by $\Del(G)$ the maximum degree of $G$. All graphs throughout are
simple.

It is well-known and easy to prove that if $d$ is the degree sequence of a
graph $G$, then, denoting by $\|d\|_2$ the $\ell^2$-norm of $d$, we have
\begin{equation}\label{e:lmax-bounds}
  \|d\|_2 \le \lmax(G) \le \Del(G);
\end{equation}
in particular, if $G$ is $r$-regular, then $\lmax(G)=r$. This basic
observation determines completely the meaning of the largest eigenvalue for
regular graphs and indeed, for all graphs which are ``nearly regular'' in the
sense that, say, the maximum degree does not exceed a constant multiple of
the average degree.

In this paper we investigate the general case, where a significant gap
between the $\ell^2$-norm of the degree sequence and the maximum degree can
potentially exist, with the ultimate goal to understand the relation between
the largest eigenvalue of a graph and its degree sequence in this case.

We notice that one can expect the largest eigenvalue to reflect mostly the
properties of the degree sequence, and to a much lesser extent the structure
of the graph itself, as there exist graphs with the same number of vertices
and the same largest eigenvalue which do not look similar --- as, for
instance, all graphs $K_{n_1,n_2}\cup{\overline K}_{n-(n_1+n_2)}$ with the
product $n_1n_2$ and integer $n>n_1+n_2$ fixed.

The last example (more generally, the fact that the largest eigenvalue is
monotonic
by the interlacing theorem, while the degree sequence is easy to manipulate,
say, by adding isolated vertices), suggests that it is insufficient for our
purposes to confine to the graph itself. Instead, we have to bring into
consideration the whole family of its induced subgraphs; more precisely, of
the induced subgraphs of its bipartite double cover. Specifically, suppose
that $G$ is a graph on the vertex set $V$, and let $X,Y\seq V$ be non-empty
sets of vertices. Consider the subgraph $G_{X,Y}$ of the bipartite double
cover $G\times K_2$, induced by $X$ and $Y$ (or rather copies thereof, taken
in different partite sets). Thus, $G_{X,Y}$ is a bipartite graph with
disjoint copies of $X$ and $Y$ as the partite sets, and with the number of
edges equal to the number of edges between $X$ and $Y$ in $G$. Using the
standard notation $e(X,Y)$ for this number of edges, the average degree of a
vertex from $X$ in $G_{X,Y}$ is $e(X,Y)/|X|$, and the average degree in
$G_{X,Y}$ of a vertex from $Y$ is $e(X,Y)/|Y|$. The geometric mean of these
averages, which is $e(X,Y)/\sqrt{|X||Y|}$, can thus be considered as a
measure of the average degree of $G_{X,Y}$. We give this measure a designated
name, defining the \emph{bi-average degree} of a bipartite graph with the
partite sets $U$ and $W$ to be $e(U,W)/\sqrt{|U||W|}$. Therefore, the
quantity
  $$ \M(G) := \max_{\est\ne X,Y\seq V} \frac{e(X,Y)}{\sqrt{|X||Y|}} $$
can be interpreted as the maximum bi-average degree of an induced subgraph of
the bipartite double cover of $G$.

We notice that for any graph $G$ we have
\begin{equation}\label{e:M-bounds}
  \od(G) \le \M(G) \le \Del(G),
\end{equation}
the lower bound following from the fact that, denoting by $V$ the vertex set
of $G$, the bi-average degree of the whole bipartite double cover $G\times
K_2$ is $\frac{e(V,V)}{|V|}=\od(G)$, and the upper bound from
  $$ e(X,Y) \le \min\big\{|X|\,\Del(G),|Y|\,\Del(G)\big\}
                          \le \Del(G)\sqrt{|X||Y|},\quad X,Y\seq V. $$
As a consequence of \refe{M-bounds}, if $G$ is $r$-regular, then $\M(G)=r$.

With this notation, we can state our main results (to be proved in subsequent
sections).
\begin{theorem}\label{t:main-general}
If $G$ is a graph with the vertex set $V$, then
  $$ \textstyle
         \M(G) \le \lmax(G) \le \big(\frac14\log|V|+1\big)\,\M(G). $$
\end{theorem}

By the remark above, and since $\lmax(G)=r$ for an $r$-regular graph $G$, the
lower bound of Theorem \reft{main-general} is attained if $G$ is regular. As
to the upper bound, we have no reasons to believe that it is sharp. However,
in Section \refs{example} we construct a sequence of graphs $G$ of
arbitrarily large order $n$ such that
  $$ \M(G)
        \ll \left(\frac{\log\log n}{\log n}\right)^{1/8} \,\lam_{\max}(G) $$
(with an absolute implicit constant); this shows that the upper bound of
Theorem \reft{main-general} cannot be improved all the way down to the lower
bound. Our construction uses a version of the tensor power trick and its
analysis is rather complicated technically; finding a simpler construction
would be interesting.

We notice that the close relation between the quantity $\M(G)$ and the
largest eigenvalue $\lmax(G)$ stems form the fact that if $A$ denotes the
adjacency matrix of $G$, and $n$ is the order of $G$, then
\begin{equation}\label{e:M-explained}
  \M(G) = \max_{0\ne\xi,\eta\in\{0,1\}^n} \frac{\xi^t A\eta}{\|\xi\|\|\eta\|}
\end{equation}
(as it follows by associating to every subset of the vertex set of $G$ its
characteristic vector), whereas
  $$ \lmax(G) = \sup_{0\ne x,y\in\R^n} \frac{x^tAy}{\|x\|\|y\|}. $$
This observation immediately yields the estimate
  $$ \M(G)\le\lmax(G), $$
which was included into the statement of Theorem \reft{main-general} (and
will be included also into subsequent theorems) just for completeness.

Our next result presents an improvement over Theorem \reft{main-general} for
sparse graphs.
\begin{theorem}\label{t:main-general-sparse}
If $G$ is a graph with the vertex set $V$, then
  $$ \textstyle
         \M(G) \le \lmax(G)
               \le \sqrt{\big(\log\Del(G)+1\big)
                                     \big(\frac14\log|V|+1\big)}\,\M(G). $$
\end{theorem}

To present a yet more robust estimate, we introduce the following notation.
Given a finite sequence $d$ with non-negative terms, consider the
non-increasing rearrangement $d_1\longge d_n\ge 0$ of the terms of $d$,
define $k\in[1,n]$ to be the smallest positive integer with
  $$ d_1^2\longp d_k^2 \ge d_{k+1}^2\longp d_n^2, $$
and let $\rho(d):=d_1/d_k$ if $d_k\ne 0$, and $\rho(d):=1$ if $d_k=0$ (in
which case $k=1$ and $d$ is the zero sequence) The quantity $\rho(d)$
measures how smooth is $d$. We record the following simple bounds:
\begin{itemize}
\item[i)]  $\rho(d)\ge 1$. (Equality is attained, for instance, if all
    positive coordinates of $d$ are equal to each other.)
\item[ii)] $\rho(d)\le\sqrt n$: for $k=1$ this is trivial, and for $k>1$
    follows from
      $$ d_1^2 \le d_1^2 \longp d_{k-1}^2
                                   < d_k^2 \longp d_n^2 < nd_k^2. $$ (On
    the other hand, if $d_1=\sqrt n$ and $d_2\longe d_n=1+1/n$, then
    $\rho(d)=(1-o(1))\sqrt n$ as $n\to\infty$.)
\item[iii)] $\rho(d)<\sqrt2\ \|d\|_\infty/\|d\|_2$ (with $\|\cdot\|_p$
    denoting the $\ell^p$-norm): to see this, notice that
    $\|d\|_\infty=d_1$ and
      $$ nd_k^2 \ge d_k^2\longp d_n^2 > \frac n2\,\|d\|_2^2. $$
\end{itemize}

\begin{theorem}\label{t:master}
Let $G$ be a graph with the vertex set $V$. For each $v\in V$ and $X\seq V$,
denote by $d_X(v)$ the number of neighbors of $v$ in $X$, and let
 $K:=\max_{\est\ne X\seq V} \rho((d_X(v))_{v\in V})$. Then
  $$ \textstyle
     \M(G) \le \lmax(G) \le \sqrt{2(\log K+1)(\log|V|+4)}\, \M(G). $$
\end{theorem}

Clearly, for any sequence $d$ with integer terms we have
$\rho(d)\le\|d\|_\infty$. Consequently, Theorem \reft{master} readily implies
Theorem \reft{main-general-sparse}, albeit with slightly weaker constants.

\section{Background and summary of results: bipartite graphs}

Theorems \reft{main-general}--\reft{master} can be refined in the situation
where the graph $G$ under consideration is bipartite. Indeed, the very
definition of the quantity $\M(G)$ can be given a cleaner shape in this case.
\begin{claim}\label{c:bipartiteM}
If $G$ is a bipartite graph with the partite sets $U$ and $W$, then
  $$ \M(G) = \max_{\sub{\est\ne X\seq U \\ \est\ne Y\seq W}}
                                        \frac{e(X,Y)}{\sqrt{|X||Y|}}. $$
\end{claim}

\begin{proof}
It suffices to show that for any $\est\ne X_U,Y_U\seq U$ and
 $\est\ne X_W,Y_W\seq W$ we have
  $$ \frac{e(X_U\cup X_W,Y_U\cup Y_W)}{\sqrt{(|X_U|+|X_W|)(|Y_U|+|Y_W|)}}
        \le \max \left\{ \frac{e(X_U,Y_W)}{\sqrt{|X_U||Y_W|}},
                         \frac{e(X_W,Y_U)}{\sqrt{|X_W||Y_U|}} \right\}. $$
To this end we denote by $T$ the maximum in the right-hand side, so that
$e(X_U,Y_W)\le T\sqrt{|X_U||Y_W|}$ and
 $e(X_W,Y_U)\le T\sqrt{|X_W||Y_U|}$, and observe that then
\begin{align*}
  e(X_U\cup X_W,Y_U\cup Y_W)
    &=   e(X_U,Y_W)+e(X_W,Y_U) \\
    &\le T\big( \sqrt{|X_U||Y_W|} + \sqrt{|X_W||Y_U|} \big) \\
    &\le T \sqrt{(|X_U|+|X_W|)(|Y_U|+|Y_W|)}.
\end{align*}
\end{proof}

The following corollary will be used in conjunction with the fact that
the largest eigenvalue of the bipartite double cover of a graph is equal to
the largest eigenvalue of the graph itself.
\begin{corollary}\label{c:M=M}
For any graph $G$ we have $\M(G\times K_2)=\M(G)$.
\end{corollary}

To prove the corollary denote by $\phi'$ and $\phi''$ the
adjacency-preserving bijections of the vertex set $V$ of $G$ onto the partite
sets of $G\times K_2$, and notice that then, for any $X,Y\seq V$,
  $$ \frac{e_G(X,Y)}{\sqrt{|X||Y|}}
       = \frac{e_{G\times K_2}(\phi'(X),\phi''(Y))}%
                                            {\sqrt{|\phi'(X)| |\phi''(Y)|}} $$
(where $e_G$ and $e_{G\times K_2}$ denote the number of edges in the
corresponding graphs).

The bipartite analogue of \refe{lmax-bounds} is given by
\begin{lemma}\label{l:lmax-bipartite}
If $G$ is a bipartite graph with the partite sets $U$ and $W$ then, denoting
by $d_U$ and $d_W$ the degree sequences of $U$ and $W$, respectively, and
letting $\Del_U:=\|d_U\|_\infty$ and $\Del_W:=\|d_W\|_\infty$, we have
  $$ \textstyle
     \max\left\{ \sqrt{\frac{|U|}{|W|}}\,\|d_U\|_2,
                    \sqrt{\frac{|W|}{|U|}}\,\|d_W\|_2 \right\}
                                     \le \lmax(G) \le \sqrt{\Del_U\Del_W}. $$
Consequently,
  $$ \sqrt{\|d_U\|_2\|d_W\|_2}\le\lmax(G)\le\sqrt{\Del_U\Del_W}, $$
and, therefore, if $G$ is $(r_U,r_W)$-regular, then $\lmax=\sqrt{r_Ur_W}$.
\end{lemma}

\begin{proof}
Let $A$ denote the adjacency matrix of $G$. If $\xi\in\R^{|U|+|W|}$ is the
characteristic vector of $U$, then the non-zero coordinates of the vector
$A\xi$ form the sequence $d_W$, and therefore $\|A\xi\|=\sqrt{|W|}\|d_W\|_2$.
Hence,
  $$ \lmax(G) = \sup_{0\ne x\in\R^{|U|+|W|}} \frac{\|Ax\|}{\|x\|}
          \ge \frac{\|A\xi\|}{\|\xi\|} = \sqrt{\frac{|W|}{|U|}}\,\|d_W\|_2, $$
and in an identical way we obtain the estimate
$\lmax(G)\ge\sqrt{|U|/|W|}\,\|d_U\|_2$.

For the upper bound, suppose that $(\xi_u,\eta_w)_{u\in U,\,w\in W}$ is an
eigenvector of $A$, correspon\-ding to the eigenvalue $\lmax(G)$; thus,
\begin{equation}\label{e:lmax-bip-upper}
  \sum_{w\sim u} \eta_w = \lmax(G)\, \xi_u
             \quad\text{and}\quad \sum_{u\sim w} \xi_u = \lmax(G)\, \eta_w,
\end{equation}
with the summation in the first sum extending over all vertices $w\in W$
adjacent to the given vertex $u\in U$, and the summation in the second sum
over all vertices $u\in U$ adjacent to the given vertex $w\in W$. Letting
  $$ \xi_{\max} := \max_{u\in U} |\xi_u|
                \quad\text{and}\quad \eta_{\max} := \max_{w\in W} |\eta_w|, $$
we conclude that
  $$ \lmax(G)\, \xi_{\max} \le \Del_U\, \eta_{\max}
       \quad\text{and}\quad \lmax(G)\, \eta_{\max} \le \Del_W\, \xi_{\max}, $$
and the result follows by multiplying out the two estimates and observing
that $\xi_{\max}\eta_{\max}\ne 0$. (If we had, say, $\xi_{\max}=0$, this
would imply $\xi_u=0$ for each $u\in U$ and, consequently, $\eta_w=0$ for
each $w\in W$ by \refe{lmax-bip-upper}.)
\end{proof}

The bipartite analogue of \refe{M-bounds} is as follows: if
$G,U,W,d_U,d_W,\Del_U$, and $\Del_W$ are as in Lemma~\refl{lmax-bipartite},
then, letting $\od_U:=\|d_U\|_1$ and $\od_W:=\|d_W\|_1$, we have
\begin{equation}\label{e:M-bounds-bipartite}
     \sqrt{\od_U\od_W} \le \M(G) \le \sqrt{\Del_U\Del_W}.
\end{equation}
For the proof it suffices to notice that, on the one hand,
  $$ \M(G) \ge \frac{e(U,W)}{\sqrt{|U||W|}} = \sqrt{\frac{|U|}{|W|}}\,\od_U
                                            = \sqrt{\frac{|W|}{|U|}}\,\od_W, $$
and, on the other hand, for any $X\seq U$ and $Y\seq W$,
  $$ e(X,Y) \le \min \{ |X|\Del_U, |Y|\Del_W \}
                                   \le \sqrt{|X||Y|} \, \sqrt{\Del_U\Del_W}. $$

Notice that, as a result of \refe{M-bounds-bipartite} and Lemma
\refl{lmax-bipartite}, for an $(r_U,r_W)$-regular graph $G$ we have
  $$ \M(G)=\lmax(G)=\sqrt{r_Ur_W}. $$

We now state the bipartite versions of Theorems
\reft{main-general}--\reft{master}.

\renewcommand{\theprimetheorem}{\reft{main-general}$'$}
\begin{primetheorem}\label{p:main-general-bip}
If $G$ is a bipartite graph with the partite sets $U$ and $W$, then
  $$ \textstyle
     \M(G) \le \lmax(G) \le
         \sqrt{\big(\frac14\log|U|+1\big)\big(\frac14\log|W|+1\big)}
                                                                \,\M(G). $$
\end{primetheorem}

\renewcommand{\theprimetheorem}{\reft{main-general-sparse}$'$}
\begin{primetheorem}\label{p:main-general-sparse-bip}
If $G$ is a bipartite graph with the partite sets $U$ and $W$, then, denoting
by $\Del_U$ the maximum degree of a vertex from $U$, we have
   $$ \textstyle
      \M(G) \le \lmax(G) \le
        \sqrt{\big(\log\Del_U+1\big)\big(\frac14\log|W|+1\big)}\, \M(G). $$
\end{primetheorem}

Observing that in a bipartite graph the degree of a vertex from one partite
set does not exceed the size of another partite set, we get the following
corollary (to be compared with Theorem~\refp{main-general-bip}).
\begin{corollary}
If $G$ is a bipartite graph with the partite sets $U$ and $W$, then, letting
$n:=\min\{|U|,|W|\}$, we have
   $$ \textstyle
      \M(G) \le \lmax(G) \le \big(\frac12\log n+2\big)\,\M(G). $$
\end{corollary}

\renewcommand{\theprimetheorem}{\reft{master}$'$}
\begin{primetheorem}\label{p:master-bip}
Let $G$ be a bipartite graph with the partite sets $U$ and $W$. For each
$u\in U$ and $Y\seq W$, denote by $d_Y(u)$ the number of neighbors of $u$ in
$Y$, and let $K:=\max_{\est\ne Y\seq W} \rho((d_Y(u))_{u\in U})$. Then
  $$ \textstyle
     \M(G) \le \lmax(G) \le \sqrt{2(\log K+1)(\log|W|+4)}\, \M(G). $$
\end{primetheorem}

Theorems \reft{main-general}--\reft{master} follow immediately from Theorems
\refp{main-general-bip}--\refp{master-bip} using the following simple scheme:
given a graph $G$, apply the appropriate theorem about bipartite graphs to
the bipartite double cover $G\times K_2$, and then use Corollary~\refc{M=M}
along with the fact that $\lmax(G)=\lmax(G\times K_2)$ to return back to the
original graph $G$. For this reason, from now on we concentrate exclusively
on the proofs of Theorems \refp{main-general-bip}--\refp{master-bip}. In the
next section we state three lemmas needed for the proofs, and deduce the
theorems from the lemmas. The lemmas are proved in Section
\refs{proofs-of-the-lemmas}. In Section \refs{example} we give an example
which sets the limit to potentially possible improvements in Theorems
\reft{main-general}--\refp{master-bip}; namely, we construct graphs $G$ of
arbitrarily large order $n$ such that
\begin{equation}\label{e:lmax-M}
  \lmax(G) \gg \left(\frac{\log n}{\log\log n}\right)^{1/8}\, \M(G)
\end{equation}
(with an absolute implicit constant).

\section{Auxiliary Lemmas and Deduction of
  Theorems \refp{main-general-bip}--\refp{master-bip}}\label{s:the-three-lemmas}

The three lemmas stated below in this section show that no vector is ``almost
orthogonal'' simultaneously to all vertices of the unit cube $\{0,1\}^n$;
equivalently, there is no hyperplane to which all vertices of the unit cube
are close simultaneously. Albeit slightly technical, these three lemmas are
in the heart of our argument. Once the lemmas are stated, we show how
Theorems \refp{main-general-bip}--\refp{master-bip} follow from them. The
lemmas themselves are proved in the next section.

By $\|\cdot\|$ we denote the usual Euclidean norm on a finite-dimensional
real vector space. The standard inner product is denoted by
$\<\cdot,\cdot\>$. Thus, for instance, for an integer $n\ge 1$ and a vector
$x\in\R^n$, we have $\<x,x\>=\|x\|^2=n\|x\|_2^2$.

\begin{lemma}\label{l:ucube1}
Let $n\ge 1$ be an integer. For any vector $z\in\R^n$ with non-negative
coordinates, there exists a non-zero vector $\del\in\{0,1\}^n$ such that
  $$ \<z,\del\> \ge \frac{2}{\sqrt{\log n+4}}\,\|z\|\|\del\|. $$
\end{lemma}

Notice that the estimate of Lemma \refl{ucube1} is tight for $n=1$. For a
less trivial example, consider the vector $z:=(1,1/\sqrt2\longc1/\sqrt n)$,
and notice that for any non-zero $\del\in\{0,1\}^n$ one has
 $\<z,\del\><(2/\sqrt{\log n})\,\|z\|\|\del\|$.

\begin{lemma}\label{l:ucube2}
Let $n,\Del\ge1$ be integers. For any integer vector $z\in[0,\Del]^n$ there
exists a non-zero vector $\del\in\{0,1\}^n$ such that
  $$ \<z,\del\> \ge \frac1{\sqrt{\log\Del+1}}\,\|z\|\|\del\|. $$
\end{lemma}

For our next lemma the reader may need to recall the definition of the
function $\rho$ introduced immediately after the statement of Theorem
\reft{main-general-sparse}.
\begin{lemma}\label{l:ucube3}
Let $n\ge 1$ be an integer. For any vector $z\in\R^n$ with non-negative
coordinates, there exists a non-zero vector $\del\in\{0,1\}^n$ such that
  $$ \<z,\del\> \ge \frac1{\sqrt{8(\log\rho(z)+1)}} \, \|z\|\|\del\|. $$
\end{lemma}

For a real matrix $A$, by $\|A\|$ we denote the operator norm of $A$; that
is,
  $$ \|A\| = \sup_{x\ne 0} \frac{\|Ax\|}{\|x\|}, $$
with the Euclidean norms in the numerator and the denominator in the
right-hand side. We recall
that the operator norm of a symmetric matrix is equal to its largest
eigenvalue, and that if $A$ is a block matrix of the form
 $\begin{pmatrix} 0 & B \\ B^t & 0 \end{pmatrix}$, then
$\|A\|=\|B\|=\|B^t\|$. As a result, if $G$ is a bipartite graph with the
biadjacency matrix $B$, then $\lmax(G)=\|B\|=\|B^t\|$.

We now deduce Theorems \refp{main-general-bip}--\refp{master-bip} from Lemmas
\refl{ucube1}--\refl{ucube3}.

\begin{proof}[Proof of Theorem \refp{main-general-bip}]
Write $m:=|U|$ and $n:=|W|$ and let $B$ denote the biadjacency matrix of $G$,
with rows corresponding to the elements of $U$, and columns to the elements
of $W$. Fix $x\in\R^m\stm\{0\}$ with $\|B^tx\|=\|B\|\|x\|$. Since all entries
of $B$ are non-negative, we can assume that all coordinates of $x$ are
non-negative. (If $x$ have both positive and negative coordinates, then
switching the signs of all negative coordinates yields a vector
$x'\in\R^m\stm\{0\}$ with $\|B^tx'\|/\|x'\|\ge\|B^tx\|/\|x\|$.) Hence, all
coordinates of the vector $B^tx\in\R^n$ are non-negative, too, and applying
Lemma~\refl{ucube1} to this vector, we find a non-zero vector
$\eta\in\{0,1\}^n$ so that
  $$ \<B^tx,\eta\> \ge \frac{2}{\sqrt{\log n+4}} \, \|B^tx\|\|\eta\|. $$
Since $\<B^tx,\eta\>=\<x,B\eta\>\le\|x\|\|B\eta\|$ and $\|B^tx\|=\|B\|\|x\|$,
this gives
\begin{equation}\label{e:Beta-large}
  \|B\eta\|\ge \frac{2}{\sqrt{\log n+4}}\, \|B\|\|\eta\|.
\end{equation}
Applying now Lemma \refl{ucube1} to the vector $B\eta\in\R^m$, we find a
non-zero vector $\xi\in\{0,1\}^m$ with
  $$ \<B\eta,\xi\> \ge \frac2{\sqrt{\log m+4}} \, \|B\eta\|\|\xi\|. $$
Combining this with \refe{Beta-large}, we get
  $$ \xi^t B\eta = \<B\eta,\xi\>
       \ge \frac4{\sqrt{(\log m+4)(\log n+4)}}\,\|B\|\|\xi\|\|\eta\|. $$
To complete the proof we notice that if $X\seq U$ is the subset with the
characteristic vector $\xi$, and $Y\seq W$ is the subset with the
characteristic vector $\eta$, then $|X|=\|\xi\|^2,\ |Y|=\|\eta\|^2$, and
$e(X,Y)=\xi^tB\eta=\<\xi,B\eta\>$, whence
  $$ \M(G) \ge \frac{e(X,Y)}{\sqrt{|X||Y|}}
            = \frac{\xi^tB\eta}{\|\xi\|\|\eta\|}
                      \ge \frac4{\sqrt{(\log m+4)(\log n+4)}}\,\|B\|. $$
The result now follows in view of $\|B\|=\lmax(G)$.
\end{proof}

\begin{proof}[Proof of Theorem \refp{main-general-sparse-bip}]
We act as in the proof of Theorem \refp{main-general-bip}, except that the
second application of Lemma \refl{ucube1} is replaced with an application of
Lemma \refl{ucube2}. Specifically, let $m,n,B,x$, and $\eta$ be as in the
proof of Theorem \refp{main-general-bip}, so that \refe{Beta-large} holds
true. Applying Lemma \refl{ucube2} to the vector $B\eta\in[0,\Del_U]^m$, we
find a non-zero vector $\xi\in\{0,1\}^m$ with
  $$ \<B\eta,\xi\> \ge \frac1{\sqrt{\log\Del_U+1}} \, \|B\eta\|\|\xi\|. $$
Comparing with \refe{Beta-large} we obtain
  $$ \xi^tB\eta \ge
          \frac2{\sqrt{(\log\Del_U+1)(\log n+4)}}\,\|B\|\|\xi\|\|\eta\| $$
and the rest of the argument is exactly as in the proof of Theorem
\refp{main-general-bip}.
\end{proof}

\begin{proof}[Proof of Theorem \refp{master-bip}]
We define $m,n,B,x$ and $\eta$ as in the proofs of Theorems
\refp{main-general-bip} and \refp{main-general-sparse-bip}, and this time
replace the second application of Lemma \refl{ucube1} in Theorem
\refp{main-general-bip} with an application of Lemma \refl{ucube3} to the
vector $B\eta$, to find $\xi\in\{0,1\}^m$ such that
  $$ \<B\eta,\xi\> \ge \frac1{\sqrt{8(\log K+1)}} \, \|B\eta\|\|\xi\|. $$
The proof then can be completed as those of Theorems \refp{main-general-bip}
and \refp{main-general-sparse-bip}.
\end{proof}

An important (though somewhat implicit) ingredient of the proofs of Theorems
\refp{main-general-bip}--\refp{master-bip} is the assertion that for any
matrix $B$ with non-negative entries, denoting by $n$ the number of columns
of $B$, we can find a non-zero vector $\eta\in\{0,1\}^n$ satisfying
\refe{Beta-large}. We notice that the coefficient in the right-hand side of
\refe{Beta-large} is essentially best possible, as one can easily check
taking $B$ to be the matrix of the orthogonal projection of $\R^n$ onto the
vector $(1,1/\sqrt{2}\longc 1/\sqrt{n})$. It is quite possible, however, that
this coefficient can be improved in the special case where the entries of $B$
are restricted to the values $0$ and $1$. A result of this sort would
immediately lead to an improvement in Theorems
\reft{main-general}--\refp{master-bip}.

\section{Proofs of Lemmas \refl{ucube1}--\refl{ucube3}}%
  \label{s:proofs-of-the-lemmas}

\begin{proof}[Proof of Lemma \refl{ucube1}]
We write $z=(z_1\longc z_n)$ and, without loss of generality, assume that
\begin{equation}\label{e:realg1}
  z_1 \longge z_n \ge 0\quad \text{and}\quad \|z\|=1.
\end{equation}
Let $\tau := 2/\sqrt{\log n+4}$. We will show that there exists $k\in[n]$
with $z_1\longp z_k\ge\tau\sqrt k$; choosing then $\del$ to be the vector
with the first $k$ coordinates equal to $1$ and the rest equal to $0$
completes the proof.

Suppose, for a contradiction, that $z_1\longp z_k<\tau\sqrt{k}$ for
$k=1\longc n$. Multiplying this inequality by $z_k-z_{k+1}$ for each
$k\in[n-1]$, and by $z_n$ for $k=n$, adding up the resulting estimates, and
rearranging the terms, we obtain
  $$ z_1^2\longp z_n^2 < \tau \big(z_1+(\sqrt2-1)z_2
                                      \longp(\sqrt n-\sqrt{n-1})z_n \big). $$
Using Cauchy-Schwartz and recalling \refe{realg1} gives
  $$ 1 < \tau \Big( \sum_{k=1}^n
             \big(\sqrt k-\sqrt{k-1}\big)^2 \Big)^{1/2}
                                        \le \frac12\,\tau \sqrt{\log n +4} $$
(we omit the routine estimate of the last sum), a contradiction.
\end{proof}

\begin{proof}[Proof of Lemma \refl{ucube2}]
For every $i\in[0,\Del]$, let $n_i$ denote the number of coordinates of $z$
which are equal to $i$, so that $n=n_0+n_1\longp n_\Del$ and
$\|z\|^2=n_1\longp \Del^2n_\Del$. Consider the vector $\del_i\in\{0,1\}^n$
with each coordinate being $1$ whenever the corresponding coordinate of $z$
is at least $i$, and being $0$ otherwise. We have
 $\|\del_i\|^2=n_i\longp n_\Del$ and, as a result,
  $$ \<\del_i,z\> = in_i\longp\Del n_\Del \ge i \|\del_i\|^2. $$
Consequently, if
  $$ \<\del_i,z\> < \tau \|z\|\|\del_i\| $$
holds for some $\tau>0$ and every $i\in[1,\Del]$, then
  $$ \<\del_i,z\>^2 < \tau^2 \|z\|^2\cdot\frac1i\,\<\del_i,z\>, $$
implying
  $$ i(in_i\longp \Del n_\Del) = i\<\del_i,z\>
                                  < \tau^2 \|z\|^2, \qquad i\in[1,\Del]. $$
Dividing through by $i$ and taking the sum over all $i\in[1,\Del]$ yields
  $$ \tau^2 \|z\|^2 (\log\Del+1) > \sum_{i=1}^\Del \sum_{j=i}^\Del jn_j
                                    = \sum_{j=1}^\Del j^2 n_j = \|z\|^2, $$
and the assertion follows.
\end{proof}

\begin{proof}[Proof of Lemma \refl{ucube3}]
Without loss of generality we assume that $z=(z_1\longc z_n)$ with
$z_1\longge z_n\ge 0$. Let $k\in[1,n]$ be the smallest integer with
$z_1^2\longp z_k^2\ge z_{k+1}^2\longp z_n^2$, as in the definition of the
quantity $\rho(z)$. Writing $\Del:=\lfloor\rho(z)\rfloor$ and applying Lemma
\refl{ucube2} to the vector
 $z':=(\lfl z_1/z_k\rfl\longc\lfl z_n/z_k\rfl)\in[0,\Del]^n$, we find a
non-zero $\del\in\{0,1\}^n$ so that
  $$ \<z',\del\> \ge \frac1{\sqrt{\log\Del+1}}\,\|z'\|\|\del\|. $$
It remains to notice that $\Del\le\rho(z)$, $\<z',\del\>\le\<z,\del\>/z_k$,
and
  $$ \|z'\|^2 =   \sum_{i=1}^k \lfl\frac{z_i}{z_k}\rfl^2
              >   \frac1{4z_k^2} \sum_{i=1}^k z_i^2
              \ge \frac{\|z\|^2}{8z_k^2}. $$
\end{proof}

\section{Graphs with $\M(G)=o(\lmax(G))$}\label{s:example}

Our goal in this section is to construct graphs $G$ of arbitrarily large
order with the largest eigenvalue $\lmax(G)$ exceeding considerably the
maximum bi-average degree $\M(G)$, cf.~\refe{lmax-M}. This will show that
Theorems \reft{main-general}--\refp{master-bip} are reasonably sharp.

The idea behind our construction is to take $G$ to be a graph whose adjacency
matrix $A$ has a large spectral gap, and has its Perron-Frobenius
eigenvector, say $e$, highly non-aligned with any $(0,1)$-vector. The former
property ensures that for any vector $\del$, the norm $\|A\del\|$ is
controlled by the projection of $\del$ onto $e$, and then the latter property
shows that whenever $\del$ is a $(0,1)$-vector, $\|A\del\|$ is small. This
results in $\M(G)$ being small. In practice, we take $A$ to be a high tensor
power of a matrix with a large spectral gap. The spectral gap of the original
matrix is then inherited by $A$, whereas the property of being non-aligned
with $(0,1)$-vectors, somewhat unexpectedly, is acquired by passing to tensor
powers.

Let $H$ denote the entropy function extended by continuity onto the interval
$[0,1]$; thus, $H(x)=-x\log x -(1-x)\log(1-x)$ for $x\in(0,1)$, and
$H(0)=H(1)=0$.

The following estimates are easy to derive using the Stirling formula:
\begin{equation}\label{e:binest}
  \frac13\,\frac1{\sqrt q}\, e^{tH(q/t)} < \binom tq
                < \frac23\,\frac1{\sqrt q}\, e^{tH(q/t)},\quad 1\le q\le t/2.
\end{equation}
We also need the following large deviation inequality.
\begin{lemma}\label{l:large-deviation}
For any real $\lam>0$ and positive integer $q$ and $t$ with $q\le
t/(\lam+1)$, we have
  $$ \sum_{j=0}^q \binom tj \lam^{t-j}  \le \lam^{t-q} e^{tH(q/t)}. $$
\end{lemma}

\begin{proof}
Dividing through both sides of the inequality by $\lam^{t-q}$, we get an
increasing function of $\lam$ in the left-hand side and a quantity,
independent of $\lam$, in the right-hand side. Therefore, the general case
will follow from that where $q=t/(\lam+1)$, which we now assume to hold. For
brevity we write $p:=q/t$, so that $\lam+1=p^{-1}$ and $\frac{1-p}p=\lam$.
The left-hand side of the inequality in question can now be estimated from
above by
    $$ (\lam+1)^t = p^{-t} = \lam^{(1-p)t} p^{-pt}(1-p)^{-(1-p)t}
                                                = \lam^{t-q} e^{tH(q/t)},  $$
as wanted.
\end{proof}

We remark that, despite its seemingly vacuous proof, the estimate of Lemma
\refl{large-deviation} is surprisingly sharp: say, numerical computations
suggest that for any $q,t$, and $\lam$, the right-hand side of the inequality
of the lemma is at most twice larger than its left-hand side.

The reader is urged to compare our next lemma against Lemmas
\refl{ucube1}--\refl{ucube3}.
\begin{lemma}\label{l:tensor-optimization}
For real $\lam\ge 4$ and integer $s,t\ge 1$, write $n=(2s)^t$, and suppose
that $z\in\R^n$ is a vector with $\binom tj s^t$ coordinates equal to
$\lam^j$ for each $j\in[0,t]$. Then for every $\del\in\{0,1\}^n$ we have
  $$ \<z,\del\> \le \frac{4\lam}{\sqrt[4]{t}}\,\|z\|\|\del\|. $$
\end{lemma}

\begin{proof}
Observing that $\|\del\|^2$ is the number of coordinates of $\del$, equal to
$1$, and writing
  $$ \|\del\|^2 = \sum_{j=0}^q \binom tj s^t + r, $$
we have to show that
  $$ \sum_{j=0}^q \binom tj s^t \lam^{t-j} + r\lam^{t-q-1}
        \le \frac{4\lam}{\sqrt[4]{t}} \, \|z\|
                                  \, \sqrt{\sum_{j=0}^q \binom tj s^t + r} $$
for all $0\le q<t$ and $0\le r\le\binom t{q+1} s^t$. For a suitable choice of
$K_1,K_2$, and $K_3$ (depending on $\lam,s,t$, and $q$), this inequality can
be re-written as
  $$ \frac{r+K_1}{\sqrt{r+K_2}} \le K_3. $$
Denoting the left-hand side by $f(r)$, we have
  $$ f'(r) = \frac{r+(2K_2-K_1)}{2(r+K_2)^{3/2}}. $$
Consequently, either $f(r)$ is monotonic on any given closed interval, or it
is decreasing on some initial segment of the interval and then increasing on
the remaining segment. In any case, the maximum value of $f$ on the interval
is attained at one of its endpoints. Hence, without loss of generality, we
can focus on the case where $r=0$ or $r=\binom t{q+1} s^t$; in other words,
it suffices to prove that
  $$ \sum_{j=0}^q \binom tj s^t \lam^{t-j}
        \le \frac{4\lam}{\sqrt[4]{t}} \, \|z\|
                  \, \sqrt{\sum_{j=0}^q \binom tj s^t};\quad 0\le q\le t. $$
Observing that $\|z\|^2=s^t(\lam^2+1)^t$ and setting
  $$ S_q := \sum_{j=0}^q \binom tj \lam^{t-j}
                           \ \text{and}\ \sig_q := \sum_{j=0}^q \binom tj, $$
we further rewrite the inequality to be proved as
\begin{equation}\label{e:toprove}
  S_q \le \frac{4\lam}{\sqrt[4]{t}} \, (\lam^2+1)^{t/2}\, \sqrt{\sig_q};
        \quad 0\le q\le t.
\end{equation}

Since $S_0=\lam^t$ and $\sig_0=1$, we have
\begin{multline*}
  \left( \frac{4\lam}{\sqrt[4]t}\, (\lam^2+1)^{t/2} \sqrt{\sig_0}
                                                       \, S_0^{-1} \right)^2
     \ge \frac{16\lam^2}{\sqrt t}\, (1+\lam^{-2})^t  \\
         \ge \frac{16\lam^2}{\sqrt t}\,(1+t\lam^{-2})
            = 16\lam\,\big(\lam t^{-1/2} + \lam^{-1}t^{1/2}\big)
               \ge 32\lam > 1.
\end{multline*}
This establishes the case where $q=0$, and we assume below that $q\ge 1$.

Let $\kap:=q/t$. We proceed by cases, splitting the interval $(0,1]$ as
  $$ \textstyle
            \left(0,\frac1{2\lam}\right]
       \cup \left[\frac1{2\lam},\frac1{\lam+1}\right]
       \cup \left[\frac1{\lam+1},\frac12\right]
       \cup \left[\frac12,1\right] $$
and considering the subinterval into which $\kap$ falls.

\medskip\noindent 1)\,
Suppose first that
\begin{equation}\label{e:optcase1}
  0 < \kap \le \frac1{2\lam}.
\end{equation}
In this case, for each $j\le q$ we have
  $$ \frac{\binom t{j-1} \lam^{t-(j-1)}}{\binom tj \lam^{t-j}}
       = \frac{\lam j}{t-j+1} < \frac{\lam}{\kap^{-1}-1}
           \le \frac{\lam}{2\lam-1} = 1-\eta^{-1}, $$
where
\begin{equation}\label{e:optcase1loc2}
  \eta = \frac{2\lam-1}{\lam-1}.
\end{equation}
Consequently,
  $$ S_q \le \binom tq \lam^{t-q} \sum_{i=0}^\infty (1-\eta^{-1})^i =
                                              \eta \binom tq \lam^{t-q}, $$
and since $\sig_q\ge\binom tq$, it suffices to show that
  $$ \eta^2 \binom tq \lam^{2t-2q}
                            \le \frac{16\lam^2}{\sqrt t}\,(\lam^2+1)^t. $$
Using \refe{binest} and observing that $\eta<4$ (as it follows from
\refe{optcase1loc2} and the assumption $\lam\ge 4$), this can be further
reduced to
  $$ \frac{1}{\sqrt{\kap}}\,e^{tH(\kap)} \lam^{2t(1-\kap)}
                                                 \le \lam^2\,(\lam^2+1)^t, $$
and, by passing to logarithms, dividing through by $t$, and rearranging the
terms, to
  $$ \log(\lam^2+1) - 2(1-\kap)\log\lam
                                \ge H(\kap) - \frac1{2t}\log(\lam^4\kap) . $$
Optimizing by $\lam$, it is not difficult to see that
$\log(\lam^2+1)-2(1-\kap)\log\lam\ge H(\kap)$ for all $\lam>0$. This settles
the case where $\lam^4\kap\ge1$, and it remains to consider the situation
where $\kap\le\lam^{-4}$. Since $\kap=q/t\ge t^{-1}$, it suffices to prove
that in this case
  $$ \log(\lam^2+1) - 2(1-\kap)\log\lam
                              \ge H(\kap) - \frac1{2}\kap\log(\lam^4\kap); $$
equivalently,
\begin{equation}\label{e:optcase1loc1}
  \log(\lam^2+1) - 2(1-2\kap)\log\lam \ge H(\kap) - \frac1{2}\kap\log(\kap).
\end{equation}
Since $\log(\lam^2+1)-2(1-2\kap)\log\lam$ is a decreasing function of $\lam$
in the range $4\le\lam\le\kap^{-1/4}$, its minimum value in this range is
  $$ \log(\kap^{-1/2}+1) + \frac12\,\log\kap - \kap\log\kap. $$
Hence \refe{optcase1loc1} will follow from
  $$ \log(\kap^{-1/2}+1) + \frac12\,\log\kap - \kap\log\kap
                                   \ge H(\kap) - \frac1{2}\kap\log(\kap), $$
which simplifies to
  $$ \log(1+\kap^{1/2}) + (1-\kap)\log(1-\kap)
                                           + \frac12\,\kap\log\kap \ge 0, $$
and in this form immediately follows from the fact that the left-hand side is
an increasing function of $\kap$ on the interval $\kap\in(0,1/8]$; hence on
the interval \refe{optcase1}.

\medskip\noindent 2)\,
Next, suppose that
\begin{equation}\label{e:optcase2}
  \frac1{2\lam} \le \kap \le \frac1{\lam+1}
\end{equation}
and, as a result,
\begin{equation}\label{e:optcase2loc1}
  \lam \ge (2\kap)^{-1} \ge \sqrt{\kap^{-1}-1}.
\end{equation}

By Lemma \refl{large-deviation} and  \refe{binest}, we have
  $$ S_q\le\lam^{(1-\kap)t}e^{tH(\kap)} \ \text{and}\ %
              \sig_q\ge\binom tq \ge \frac1{3\sqrt{\kap t}}\,e^{tH(\kap)}. $$
Thus, in view of \refe{optcase2loc1}, the result will follow from
  $$ \lam^{(1-\kap)t}e^{tH(\kap)} \le \frac{\kap^{-5/4}}{\sqrt{t}}
                                 \,(\lam^2+1)^{t/2} e^{\frac12\,tH(\kap)}. $$
By passing to logarithms, dividing through by $t/2$, and rearranging the
terms, this reduces to
  $$ \log(\lam^2+1)-2(1-\kap)\log\lam \ge
                                        H(\kap) + t^{-1}\log(\kap^{5/2}t). $$
Since the expression in the left-hand side is an increasing function of
$\lam$ in the range $\lam\ge\sqrt{\kap^{-1}-1}$, using \refe{optcase2loc1} we
get
\begin{align*}
  \log(\lam^2+1)-2(1-\kap)\log\lam
        &\ge \log(1+4\kap^2)-2\log(2\kap) + 2(1-\kap) \log(2\kap) \\
        &=   \log(1+4\kap^2) - 2\kap\log(2\kap).
\end{align*}
Also,
  $$ t^{-1}\log(\kap^{5/2}t) \le \frac1{e}\,\kap^{5/2} $$
for any $t>0$. Consequently, it suffices to show that
  $$ \log(1+4\kap^2) - 2\kap\log(2\kap)
                             \ge H(\kap) + \frac1{e}\,\kap^{5/2}, $$
and a routine investigation confirms that this holds true for all $\kap\le
1/5$, and therefore for all $\kap$ in the range \refe{optcase2}.

\medskip\noindent 3)\,
Next, suppose that
\begin{equation}\label{e:optcase3}
  \frac1{\lam+1}\le\kap\le \frac12.
\end{equation}
Using the trivial estimates $S_q\le(\lam+1)^t$ and $\sig_q\ge\binom tq$, the
latter in conjunction with \refe{binest}, in this case we reduce
\refe{toprove} to
  $$ (\lam+1)^t \le \frac{2\lam}{\sqrt t\sqrt[4]{\kap}}
                                \,(\lam^2+1)^{t/2} e^{\frac12\,tH(\kap)}, $$
and further to
\begin{equation}\label{e:optilemmaloc5}
  2\log(\lam+1)-\log(\lam^2+1)
                             \le H(\kap) - \frac1t\log(\kap^{1/2}t/(4\lam^2)).
\end{equation}
By \refe{optcase3} we have $\lam\ge\max\{\kap^{-1}-1,4\}$, and in this range
the expression in the left-hand side of \refe{optilemmaloc5} is easily seen
to be a decreasing function of $\lam$. As a result, we have
\begin{align*}
  2\log(\lam+1)-\log(\lam^2+1)
         &\le \log\frac{(1+(\kap^{-1}-1))^2}{1+(\kap^{-1}-1)^2} \\
         &=   -\log(2\kap^2-2\kap+1)\quad\text{if}\ \kap\le\frac15,
\intertext{and}
  2\log(\lam+1)-\log(\lam^2+1)
         &\le \log\frac{25}{17} \quad\text{if}\ \kap\ge\frac15.
\end{align*}
Since, on the other hand,
  $$ \frac1t\log(\kap^{1/2}t/(4\lam^2))
      \le \frac1{4\lam^2e}\,\kap^{1/2} \le \frac1{64e}\,\kap^{1/2}, $$
it suffices to show that
  $$ H(\kap) - \frac1{64e}\,\kap^{1/2} \ge
       \begin{cases}
          -\log(2\kap^2-2\kap+1) \ &\text{if}\ \kap\le\frac15, \\
          \log\frac{25}{17}      \ &\text{if}\ \frac15\le\kap\le\frac12.
           \end{cases} $$
Again, this can be verified by a straightforward computation.

\medskip\noindent 4)\,
Finally, suppose that $\kap\ge 1/2$. In this case we have
 $\sig_q\ge 2^{t-1}$; hence, in view of $S_q\le(\lam+1)^t$, it suffices to
show that
  $$ (\lam+1)^t \le \frac1{\sqrt[4]{t}}\,(2(\lam^2+1))^{t/2}. $$
This can be equivalently rewritten as
  $$ \log(2(\lam^2+1)) - 2\log(\lam+1) \ge \frac1{2t}\,\log t, $$
and the last inequality is immediate from the fact that its right-hand side
does not exceed $1/(2e)$, while the left-hand side is an increasing function
of $\lam$ in the range $\lam\ge 4$, and its value at $\lam=4$ is
$\log(34/25)>1/(2e)$.
\end{proof}

We can now complete our construction of graphs $G$ with $\M(G)$ small (as
compared to $\lmax(G)$).

For integer $s,t\ge 1$, denote by $I_s$ the identity matrix, and by $J_s$ the
all-$1$ matrix of order $s$, and let
  $$ A_s^{(t)}
       := \begin{pmatrix} J_s-I_s & I_s \\ I_s & 0 \end{pmatrix}^{\otimes t}; $$
thus, $A_s^{(t)}$ is a symmetric $(0,1)$-matrix of order $(2s)^t$, with
zeroes on the main diagonal.

It is not difficult to check that the minimal polynomial of $A_s^{(1)}$
is $(x^2-(s-1)x-1)(x^2+x-1)$. Letting $\lambda:=(s-1+\sqrt{(s-1)^2+4})/2$
(the largest root of the polynomial $x^2-(s-1)x-1$), we conclude that the
largest eigenvalue of $A_s^{(1)}$ is equal to $\lam$, while all other
eigenvalues do not exceed $(1+\sqrt 5)/2$ in absolute value. Also, it is
readily verified that the eigenvector corresponding to $\lam$ is
$e:=(\lam\longc\lam,1\longc1)$, with the $2s$ coordinates split evenly
between the values $\lam$ and $1$.

Write $n:=(2s)^t$, and let $\lam_1\longc\lam_n$ denote the eigenvalues of
$A_s^{(t)}$, with $\lam_1$ being the largest eigenvalue. Fix an
orthonormal basis $\{e_1\longc e_n\}$ of $\R^n$ such that $e_i$ is an
eigenvector, corresponding to the eigenvalue $\lam_i$. Since $A_s^{(t)}$
is the $t$th tensor power of $A_s^{(1)}$, we have $\lam_1=\lam^t$, and
$|\lam_i|\le\lam^{t-1}(1+\sqrt 5)/2$ for $2\le i\le n$. Consequently, for
any $\del\in\R^n$,
\begin{align*}
  \|A_s^{(t)}\del\|^2
     &= \lam_1^2 \<\del,e_1\>^2 + \lam_2^2 \<\del,e_2\>^2
                                             \longp \lam_n^2 \<\del,e_n\>^2 \\
     &\ll \lam^{2t} \<\del,e_1\>^2 + \lam^{2(t-1)}
                            \big(\<\del,e_2\>^2 \longp \<\del,e_n\>^2 \big) \\
     &\ll \lam^{2t} \<\del,e_1\>^2 + \frac1{s^2}\, \lam^{2t} \|\del\|^2.
\end{align*}
Since $e_1$ is proportional to the vector $e^{\otimes t}$ having
 $\binom tj\,s^{t}$ coordinates equal to $\lam^j$ for each $j\in[0,t]$, by
Lemma \refl{tensor-optimization} for any $\del\in\{0,1\}^n$ we have
  $$ \<\del,e_1\> \le \frac{4\lam}{\sqrt[4]{t}}\, \|\del\|; $$
as a result,
  $$ \|A_s^{(t)}\del\| \ll \lam^t
                \left( \frac{\lam}{\sqrt[4]{t}} + \frac1s \right) \|\del\|. $$
Observing that $\|A_s^{(t)}\|=\|A_s^{(1)}\|^t=\lam^t$ and choosing $t=s^8$ to
optimize, we get
  $$ \|A_s^{(t)}\del\|
        \ll \left(\frac{\log\log n}{\log n}\right)^{1/8}
                             \,\|A_s^{(t)}\|\|\del\|,\quad \del\in\{0,1\}^n $$
(with an absolute implicit constant).

If we now define $G$ to be the graph of order $n$ with the adjacency matrix
$A_s^{(t)}$, then by~\refe{M-explained},
\begin{align*}
  \M(G) &=   \max_{0\ne\xi,\eta\in\{0,1\}^n}
                         \frac{\<\xi,A_s^{(t)}\eta\>}{\|\xi\|\|\eta\|} \\
        &\le \max_{0\ne\eta\in\{0,1\}^n}
                         \frac{\|A_s^{(t)}\eta\|}{\|\eta\|} \\
        &\ll \left(\frac{\log\log n}{\log n}\right)^{1/8} \,\|A_s^{(t)}\| \\
        &=   \left(\frac{\log\log n}{\log n}\right)^{1/8} \,\lam_{\max}(G),
\end{align*}
as wanted.

\vfill



\end{document}